\documentclass[final]{siamltex}

\usepackage{graphicx}
\usepackage{amssymb}
\usepackage{epstopdf}
\usepackage{mathtools}
\usepackage{mathrsfs}
\usepackage{stmaryrd}
\usepackage{lipsum}

\newtheorem{remark}[theorem]{Remark}

\newcommand{\hf}{\frac{1}{2}}

\title{Discontinuous Galerkin method for fractional convection-diffusion equations}

\author{Q. Xu\thanks{School of Mathematics and Statistics, Central South University,410083 Changsha, China;Division of Applied Mathematics, Brown University,  RI, 02912 USA. Email: {\tt qw\_xu@hotmail.com}}
        \and  J. S. Hesthaven\thanks{Division of Applied Mathematics, Brown University,  RI, 02912 USA. Email: {\tt Jan.Hesthaven@Brown.edu}}}

\begin{document}

\maketitle

\begin{abstract}
We propose a discontinuous Galerkin method for convection-subdiffusion equations with a fractional operator of order $\alpha  (1<\alpha<2)$ defined through the fractional Laplacian. The fractional operator of order $\alpha$ is expressed as a composite of first order derivatives and fractional integrals of order $2-\alpha$, and the fractional convection-diffusion problem is expressed as a system of low order differential/integral equations and a local discontinuous Galerkin method scheme is derived for the equations. We prove stability and optimal order of convergence ${\cal O}(h^{k+1})$ for subdiffusion, and an order of convergence of ${\cal O}(h^{k+\hf})$ is established for the general fractional convection-diffusion problem. The analysis is confirmed by numerical examples.
\end{abstract}



\pagestyle{myheadings}
\thispagestyle{plain}
\markboth{Q. Xu and J. S. Hesthaven}{Discontinuous Galerkin method for fractional convection-diffusion equations}

\section{Introduction}
In this paper we consider the following convection diffusion equation

\begin{eqnarray}\label{problem}
\left\{
\begin{array}{ll}
\frac{\partial u(x,t)}{\partial t} + \frac{\partial}{\partial x}f(u)  = \varepsilon\left(-(-\Delta)^{\alpha/2}\right)u(x,t),\quad x\in \mathbb{R}, t\in (0,T],\\ 
u(x,0)=u_0(x), \quad x\in \mathbb{R} 
\end{array}\right.
\end{eqnarray}
where $f$ is assumed to be Lipschitz continuous and
the subdiffusion is defined through the fractional Laplacian $-(-\Delta)^{\alpha/2}(
\alpha \in(0,2])$, which can be defined using Fourier analysis
as\cite{fraclap1,fraclap2,fracNS1,fracNS2}, 

\begin{equation}
-\widehat{(-\Delta)^{\alpha/2}}u(\xi) = (2\pi)^\alpha|\xi|^\alpha\hat{u}(\xi)
\end{equation}
Equivalently, the fractional Laplacian can also be defined by a singular integral\cite{thry2,fraclap2},
\begin{equation}\label{intdef}
-(-\Delta)^{\frac{\alpha}{2}}u(x) = c_\alpha\int_{|z|>0}\frac{u(x+z)-u(x)}{|z|^{1+\alpha}}dz
\end{equation}
Equation \eqref{problem}, also known as a fractional conservation law, can be viewed as a generalization of the classical convection diffusion equation. During the last decade, it has arisen as a suitable model in many application areas,  such as geomorphology\cite{dunes,dunes2,app5}, overdriven detonations in gases\cite{detonation,app4}, signal processing\cite{app1}, anomalous diffusion in semiconductor growth \cite{app3} etc. With the special choice of $f(u)=u^2/2$, it is recognized as a fractional version of the viscous Burgers' equation. 
Fractional conservation laws, especially Burgers' equation, has been studied by many  authors from a theoretical perspective, mainly involving questions of well-posedness and regularity of the solutions\cite{thry1,thry2,thry3,thry4,thry5}. In the case of $\alpha<1$,  the solution is in general not smooth and shocks may appear even with smooth initial datum.   Similar to the classical  scalar conservation laws, an appropriate entropy formulation is needed to guarantee well-posedness\cite{thry2,thry6}. For the case $\alpha > 1$,  the existence and uniqueness of a regular solution has been proved in \cite{thry2,thry4}. In this case, the nonlocal term serves as subdiffusion term, and smooth out discontinuities from initial datum. 

Numerical studies of partial differential equations with nonlocal operator have attracted a lot of interest in recent years. Liu, Deng et al. have worked on numerical methods for fractional diffusion problems with fractional Laplacian operators or Riesz fractional derivatives \cite{numerical0,numerical1, numerical2,numerical3}.  Briani, Cont, Matache, et al. have considered numerical methods for financial model with fractional Laplacian operators\cite{numerical4,numerical5,numerical6}.   However, for conservation laws with fractional Laplacian operators, the development of accurate and robust numerical methods remains limited. Droniou is the first to analyze a general class of difference methods for fractional conservation laws\cite{firstfdm}. Azeras proposed a class of finite difference schemes for solving a fractional anti\-diffusive equation\cite{fdm2}. Bouharguane proposed a splitting methods for the nonlocal Fowler equation in \cite{splitting}.
For the case $\alpha <1$,  Cifani et al. \cite{cifani,cifani2} applied the discontinuous Galerkin method to the fractional conservation law and degenerate convection diffusion equations and developed some error estimation. Unfortunately, their numerical results failed to confirm their analysis.

The discontinuous Galerkin method is a well established method for classical conservation laws \cite{Jan}. However, for equations containing higher order spatial derivatives, discontinuous Galerkin methods cannot be directly applied\cite{LDG1,shu0}. A careless application of the discontinuous Galerkin method to a problem with high order derivatives could yield an inconsistent method.
The idea of local discontinuous Galerkin methods \cite{LDG1} for time dependent partial differential equations with higher derivatives is to rewrite the equation into a first order system and then apply the discontinuous Galerkin method to the system. A key ingredient for the success of this methods is the correct design of interface numerical fluxes. These fluxes must be designed to guarantee stability and local solvability of all the auxiliary variables introduced to approximate the derivatives of the solution. 

In this paper, we will consider fractional convection diffusion equations with a fractional Laplacian operator of order $\alpha (1<\alpha<2)$. Especially for $1<\alpha<2$, it is conceptionally similar to a fractional derivative  with an order between 1 and 2. To obtain a consistent and high accuracy method for this problem, we shall rewrite the fractional operator as a composite of first order derivatives and a fractional integral, and convert the fractional convection diffusion equation into a system of low order equations. This allows us to apply the local discontinuous Galerkin method.

This paper is organized as follows. In Section 2, we introduce some basic definitions and state a few central lemmas. In Section 3, we will derive the discontinuous Galerkin formulation for the fractional convection-diffusion law. In section 4,  we develop stability and convergence analysis for fractional diffusion and convection-diffusion equations. In Section 5 some numerical examples are carried out to support our analysis. Conclusions are offered in Section 6. 
 
\section{Background and definitions}

Apart from the definitions of the fractional Laplacian based on the Fourier and integral form above, it can also be defined using ideas of fractional calculus\cite{fraclap1,fraclap2,numerical1}, as
\begin{eqnarray}\label{fracdef}
-(-\Delta)^{\alpha/2}u(x) =\frac{\partial^\alpha}{\partial |x|^\alpha}u(x)= -\frac{{_{-\infty}D_x^\alpha}u(x) + {_xD_{\infty}^\alpha}u(x)}{2\cos\left(\frac{\alpha\pi}{2}\right)}
\end{eqnarray}
where $\prescript{}{-\infty}{D}_x^\alpha$ and $\prescript{}{x}{D}_{\infty}^\alpha$ refer to the left and right Riemann-Liouville fractional derivatives, respectively, of $\alpha$'th order. This definition is also known as a Riesz derivative. In this paper, we will based our developments and analysis mainly on this definition and will, in preparation, introduce a few definitions and properties of fractional integral/derivative to set the stage. 

The left and right fractional integrals of order $\alpha$ are defined as
\begin{equation}
{ _{-\infty}I_x^\alpha}u(x) = \frac{1}{\Gamma(\alpha)}\int_{-\infty}^x(x-t)^{\alpha-1}u(t)dt 
\end{equation}
\begin{equation}
{ _xI_{\infty}^\alpha}u(x) = \frac{1}{\Gamma(\alpha)}\int^{\infty}_x (t-x)^{\alpha-1}u(t)dt 
\end{equation}
This allows for the definition of the left and right Riemann-Liouville fractional derivative of order $\alpha$ $(n-1<\alpha < n)$ as,
\begin{equation}
{_{-\infty}D_x^\alpha}u(x) = \frac{1}{\Gamma(n-\alpha)}\left(\frac{d}{dx}\right)^n\int_{-\infty}^x(x-t)^{n-1-\alpha}u(t)dt=D^n( {_{-\infty}I_x^{n-\alpha}}u(x))
\end{equation}
\begin{equation}
{_{x}D_\infty^\alpha}u(x) = \frac{1}{\Gamma(n-\alpha)}\left(-\frac{d}{dx}\right)^n\int^{\infty}_x(t-x)^{n-1-\alpha}u(t)dt=(-D)^n( {_xI_\infty^{n-\alpha}}u(x))
\end{equation}

Mirroring this, Caputo's left and right fractional derivatives of order $\alpha$  $(n-1<\alpha < n)$ are defined as
\begin{equation}
{\prescript{C}{-\infty}D_x^\alpha}u(x) = \frac{1}{\Gamma(n-\alpha)}\int_{-\infty}^x(x-t)^{n-1-\alpha}\left(\frac{d}{dt}\right)^nu(t)dt= {_{-\infty}I_x^{n-\alpha}}(D^nu(x))
\end{equation}
\begin{equation}
{\prescript{C}{x}D_\infty^\alpha}u(x) = \frac{1}{\Gamma(n-\alpha)}\int^{\infty}_x(t-x)^{n-1-\alpha}\left(-\frac{d}{dt}\right)^nu(t)dt={_xI_\infty^{n-\alpha}}\left((-D)^nu(x)\right)
\end{equation}
The two definitions, the Riemann-Liouville fractional derivative and Caputo's fractional derivatives, are naturally related and they are equivalent provided all derivatives less than order $n$ of $u(x)$ disappear when $x\to \pm\infty$.

Fractional integrals and derivatives satisfy the following properties,

\begin{lemma}(Linearity\cite{Igor}).
\begin{eqnarray}
&&\prescript{}{-\infty}{I}_x^\alpha(\lambda f(x)+\mu g(x)) = \lambda\prescript{}{-\infty}{I}_x^\alpha f(x) + \mu\prescript{}{-\infty}{I}_x^\alpha g(x)\\
&&\prescript{}{-\infty}{D}_x^\alpha(\lambda f(x)+\mu g(x)) = \lambda\prescript{}{-\infty}{D}_x^\alpha f(x) + \mu\prescript{}{-\infty}{D}_x^\alpha g(x)
\end{eqnarray}
\end{lemma}

\begin{lemma}(Semigroup property\cite{Igor}). For $\alpha,\beta>0$
\begin{eqnarray}
&&\prescript{}{-\infty}{I}_x^{\alpha+\beta} f(x) = \prescript{}{-\infty}{I}_x^{\alpha}\left(\prescript{}{-\infty}{I}_x^{\beta}  f(x)\right)=\prescript{}{-\infty}{I}_x^{\beta}\left(\prescript{}{-\infty}{I}_x^{\alpha}  f(x)\right)
\end{eqnarray}
\end{lemma}

\begin{lemma}\label{split}\cite{Igor}
 Suppose $u^{(j)}(x) = 0$,when $x \to \pm\infty, \forall \,0\leqslant j \leqslant n$ ( $n-1< \alpha <n$), then
\begin{eqnarray}
\prescript{}{-\infty}{D^\alpha_x}u(x) = D^n\left(\prescript{}{-\infty}{I^{n-\alpha}_x}u(x)\right) = \prescript{}{-\infty}{I^{n-\alpha}_x}\left(D^nu(x)\right)\\
\prescript{}{x}{D^\alpha_\infty}u(x) = (-D)^n\left(\prescript{}{x}{I^{n-\alpha}_\infty}u(x)\right) = \prescript{}{x}{I^{n-\alpha}_\infty}\left((-D)^nu(x)\right)
\end{eqnarray}

\end{lemma}

From Lemma \ref{split}, we get, for $1<\alpha<2$ and a continuous function $u$ with $u^{(j)}(x)=0,x\to \pm\infty,0\leqslant j \leqslant 2$
\begin{eqnarray}\label{defsplit}
\nonumber-(-\Delta)^{\alpha/2}u(x)&=& -\frac{1}{2\cos(\alpha\pi/2)}\frac{d^2}{d x^2}\left(  {_{-\infty}I_x^{2-\alpha}}u(x)+ {_xI_\infty^{2-\alpha}}u(x)\right) \\
&=&-\frac{1}{2\cos(\alpha\pi/2)}\frac{d}{d x}\left(  {_{-\infty}I_x^{2-\alpha}}\frac{du(x)}{dx}+ {_xI_\infty^{2-\alpha}}\frac{du(x)}{dx}\right) \\
\nonumber &=&-\frac{1}{2\cos(\alpha\pi/2)}\left(  {_{-\infty}I_x^{2-\alpha}}\frac{d^2u(x)}{d x^2}+ {_xI_\infty^{2-\alpha}}\frac{d^2u(x)}{d x^2}\right) 
\end{eqnarray}
For $0<s<1$, we define 
\begin{eqnarray*}
\Delta_{-s/2}u(x) = -\frac{{_{-\infty}D_x^{-s}}u(x) + {_xD_{\infty}^{-s}}u(x)}{2\cos\left(\frac{(2-s)\pi}{2}\right)}=\frac{{_{-\infty}I_x^{s}}u(x) + {_xI_{\infty}^{s}}u(x)}{2\cos\left(\frac{s\pi}{2}\right)}
\end{eqnarray*}
When $1<\alpha<2$, we have
\begin{eqnarray}\label{split2}
-(-\Delta)^{\alpha/2}u(x) =\frac{d^2}{dx^2}\left(\Delta_{(\alpha-2)/2}u\right)=\Delta_{(\alpha-2)/2}\left(\frac{d^2u}{dx^2}\right)= \frac{d}{d x}\left(\Delta_{(\alpha-2)/2}\frac{du}{dx}\right)
\end{eqnarray}
To carry out the analysis, we introduce appropriate fractional spaces. \\
\begin{definition}{(Left fractional space\cite{roop1} ).}
We define the following seminorm
\begin{equation}
|u|_{J_L^{\alpha}(\mathbb{R})}=\|\prescript{}{-\infty}{D^\alpha_x}u\|_{L^2(\mathbb{R})} 
\end{equation}
and the norm
\begin{equation}
\|u\|_{J_L^{\alpha}(\mathbb{R})}=\left(|u|^2_{J_L^{\alpha}(\mathbb{R})} +\|u\|^2_{L^2(\mathbb{R})} \right)^\frac{1}{2}
\end{equation}
and let $J_L^\alpha(\mathbb{R})$ denote the closure of $C_0^\infty(\mathbb{R})$ with respect to $\|\cdot \|_{J_L^\alpha}$. \\
\end{definition}

\begin{definition}{(Right fractional space\cite{roop1} ).}
We define the following seminorm
\begin{equation}
|u|_{J_R^{\alpha}(\mathbb{R})}=\|\prescript{}{x}{D^\alpha_\infty}u\|_{L^2(\mathbb{R})} 
\end{equation}
and the norm
\begin{equation}
\|u\|_{J_R^{\alpha}(\mathbb{R})}=\left(|u|^2_{J_R^{\alpha}(\mathbb{R})} +\|u\|^2_{L^2(\mathbb{R})} \right)^\frac{1}{2}
\end{equation}
and let $J_R^\alpha(\mathbb{R})$ denote the closure of $C_0^\infty(\mathbb{R})$ with respect to $\|\cdot \|_{J_R^\alpha}$. \\
\end{definition}

\begin{definition}{(Symmetric fractional space\cite{roop1}) .}
We define the following seminorm
\begin{equation}
|u|_{J_S^{\alpha}(\mathbb{R})}=\left|\left(\prescript{}{-\infty}{D^\alpha_x}u, \prescript{}{x}{D^\alpha_\infty}u\right)_{L^2(\mathbb{R})} \right|^\frac{1}{2}
\end{equation}
and the norm
\begin{equation}
\|u\|_{J_S^{\alpha}(\mathbb{R})}=\left(|u|^2_{J_S^{\alpha}(\mathbb{R})} +\|u\|^2_{L^2(\mathbb{R})} \right)^\frac{1}{2}
\end{equation}
and let $J_S^\alpha(\mathbb{R})$ denote the closure of $C_0^\infty(\mathbb{R})$ with respect to $\|\cdot \|_{J_S^\alpha}$.
\end{definition}

Using these definitions, we recover the following result, \\
\begin{lemma}\label{innerproduct}{(Adjoint property\cite{Samko, roop1})}
\begin{equation}
\left( {_{-\infty}I_x^\alpha}u,    u \right)_{\mathbb{R}} = \left(  u,   {_{x}I_\infty^\alpha}u \right)_{\mathbb{R}} 
\end{equation}
\end{lemma}
\begin{lemma}\label{fracnorm}\cite{roop1}
\begin{equation}
\left( {_{-\infty}I_x^\alpha}u,   {_{x}I_\infty^\alpha}u \right)_{\mathbb{R}} = \cos(\alpha\pi)|u|^2_{J_L^{-\alpha}(\mathbb{R})}= \cos(\alpha\pi)|u|^2_{J^{-\alpha}_R({\mathbb{R}})}
\end{equation}
\end{lemma}
From Lemma \ref{innerproduct}-\ref{fracnorm}, we recover \\
\begin{lemma}\label{fracnorm2}For all $0<s<1$,
\begin{equation}
\left( \Delta_{-s}u,  u \right)_{\mathbb{R}} = |u|^2_{J_L^{-s}(\mathbb{R})}= |u|^2_{J_R^{-s}(\mathbb{R})}
\end{equation}
\end{lemma}
Generally, we consider the problem in a bounded domain instead of $\mathbb{R}$. Hence, we restrict the definition to the domain $\Omega=[a,b]$. \\
\begin{definition}
Define the spaces $J^{\alpha}_{L,0}(\Omega), J^{\alpha}_{R,0}(\Omega),  J^{\alpha}_{S,0}(\Omega))$  as the closures of $C_0^\infty(\Omega)$ under their  respective norms. 
\end{definition}

For these fractional spaces,  we have the following Theorem\cite{whdeng}, \\
\begin{theorem}\label{embedded}
If $-\alpha_2<-\alpha_1<0$, then $J^{-\alpha_1}_{L,0}(\Omega) ( \text{or}\, J^{-\alpha_1}_{R,0}(\Omega)\,  \text{or} \,  J^{-\alpha_1}_{S,0}(\Omega))$ is embedded into $J^{-\alpha_2}_{L,0}( \Omega) (\text{or}\,  J^{-\alpha_2}_{R,0}(\Omega)\, \text{or}\, J^{-\alpha_2}_{S,0}(\Omega))$, and $L^2(\Omega)$ is embedded into both of them. \\
\end{theorem}

\begin{lemma}\label{fracpc}(Fractional Poincar\'e-Friedrichs)\cite{roop1}
For $u\in J^\mu_{L,0}(\Omega)$, we have 
$$\|u\|_{L^2(\Omega)}\leqslant C |u|_{J^\mu_{L,0}},$$
and for $u\in J^\mu_{R,0}(\Omega)$, we have 
$$\|u\|_{L^2(\Omega)}\leqslant C |u|_{J^\mu_{R,0}}.$$
\end{lemma}
From the definition of the left and right fractional integral, we obtain the following result, \\
\begin{lemma}\label{leftright}
Suppose the fractional integral is defined in $[0, b]$. Let $g(y)=f(b-x)$, then
\begin{equation}
{_xI_b^\alpha}f(x) \overset{y=b-x}{=} {_0I_y^\alpha}g(y) 
\end{equation}
 \end{lemma}
\begin{lemma}\label{approrder}
Suppose $u(x)$ is a smooth function $\in \Omega \subset \mathbb{R}$. $\Omega_h $ is a discretization of the domain with interval width $h$, $u_h(x)$ is an approximation of $u$ in $P_h^k$. $\forall$ i, $u_h(x) \in I_i$ is  a polynomial of degree up to order $k$, and $(u,v)_{I_i} = (u_h,v)_{I_i} , \forall  v\in P^k$. 
 $k$ is the degree of the polynomial. Then for  $-1<\alpha \leqslant 0$,  we have
\begin{equation*}
\|\Delta_{\alpha/2}u(x)-\Delta_{\alpha/2} u_h(x)\|_{L_2}\leqslant Ch^{k+1}
\end{equation*}
and for $ 0 \leqslant n-1< \alpha \leqslant n, k\geqslant n$ we have
\begin{equation}
\|-(-\Delta)^{\frac{\alpha}{2}}u(x) +(-\Delta)^{\frac{\alpha}{2}} u_h(x)\|_{L_2}\leqslant Ch^{k+1-n}
\end{equation}
where $C$ is a constant independent of $h$. \\
\end{lemma}

\begin{proof} 
Case $-1<\alpha \leqslant 0$,

First, we consider the approximation error for a fractional integral $\|\prescript{}{a}I^{-\alpha}_x u(x) - \prescript{}{a}I^{-\alpha}_x u_h(x)\| $. 

First recall that $\|u(x)-u_h(x)\|_2 = \mathcal{O}(h^{k+1})$ from classical approximation theory. Suppose $x\in \Omega_{h,i}$. We only need to consider $r^i=\prescript{}{a}I^{-\alpha}_x (\mathcal{O}(h^{k+1}))$. 
\begin{eqnarray*}
r^i&= &\frac{1}{\Gamma(-\alpha)}\int_{a}^x(x-s)^{-\alpha-1} \mathcal{O}(h^{k+1}) ds\\  
& \leqslant & \frac{(b-a)^{-\alpha}\mathcal{O}(h^{k+1})}{\Gamma(1-\alpha)}
\end{eqnarray*}
Recalling Lemma \ref{leftright}, the case $-1<\alpha \leqslant 0$ is proved.

The case $0\leqslant n-1<\alpha<n$, recall that $\|\frac{\partial^n u}{\partial x^n} - \frac{\partial^n u_h}{\partial x^n}\| \leqslant Ch^{k+1-n}$ and use the result for case $-1<\alpha \leqslant 0$, we obtain,
\begin{equation}
\|-(-\Delta)^{\frac{\alpha}{2}}u(x) +(-\Delta)^{\frac{\alpha}{2}} u_h(x)\|_{L_2}\leqslant Ch^{k+1-n}
\end{equation}
This proves the Lemma.
\end{proof}
\\ \begin{lemma}{(Inverse properties\cite{inverse})}\label{inverse}
Suppose $V_h$ is  a finite element space spanned by polynomials up to degree $k$. For any $\omega_h\in V_h$, there exists a positive constant $C$ in dependent of $u_h$ and $h$, such that
\begin{eqnarray*}
\|\partial_x u_h\|\leqslant Ch^{-1}\|u_h\|,\quad \|u_h\|_{\Gamma_h}\leqslant Ch^{-\frac{1}{2}}\|u_h\|
\end{eqnarray*}
 
\end{lemma}
 
 \section{ LDG scheme for the fractional convection-diffusion equation}

Let us consider the fractional convection-diffusion equation with $1<\alpha<2$. From  \eqref{fracdef} and Lemma \ref{approrder}, we know that a direct approximation of the fractional laplacian operator is of order $k+1-n$. To obtain a high order discontinuous Galerkin scheme for the fractional derivative, we  rewrite the fractional derivative as a composite of first order derivatives and fractional integral and convert the equation to a low order system.

Following \eqref{split2}, we introduce two variables $p,q$, and  set  
\begin{eqnarray*}
q &= & \Delta_{(\alpha-2)/2}p\\
p &=&\sqrt{\varepsilon}\frac{\partial}{\partial x}u
\end{eqnarray*}
 Then, the fractional convection-diffusion problem can be rewritten as
\begin{eqnarray*}
\left\{
\begin{array}{l}
\frac{\partial u}{\partial t} + \frac{\partial}{\partial x}f(u) = \sqrt{\varepsilon}\frac{\partial}{\partial x}q \\
q =  \Delta_{(\alpha-2)/2}p\\
p =\sqrt{\varepsilon}\frac{\partial}{\partial x}u\\
\end{array}\right.
\end{eqnarray*}
Consider the equation in  $\Omega=[a,b]$. Given a partition $a=x_\frac{1}{2}<x_\frac{3}{2}<\cdots<x_{K+\frac{1}{2}}=b$, we denote the mesh by $I_j = [x_{j-\frac{1}{2}},x_{j+\frac{1}{2}}]$,  $\Delta x_j = x_{j+\frac{1}{2}}-x_{j-\frac{1}{2}}$.

We consider the solution in a polynomial space $V_h$, which is certainly embedded into the fractional space according to theorem \ref{embedded}.  The piece-wise polynomial space $V_h$ on the mesh is defined as,
$$V_h = \{v: v\in P^k(I_j), x \in I_j\}$$

We seek an approximation $(u_h,p_h,q_h)\in V_h$ to $(u,p,q)$ such that, for any $v,w,z \in V_h$, we have
\begin{eqnarray}\label{fracvar}
\left\{
\begin{array}{l}
\left(\frac{\partial u_h(x,t)}{\partial t}, v(x) \right)_{I_i} +  \left(\frac{\partial }{\partial x}f(u_h),v(x)\right)_{I_i}  = \sqrt{\varepsilon} \left(\frac{\partial q_h}{\partial x},v(x)  \right)_{I_i}\\
\left(q_h, w(x) \right)_{I_i}  =  \left( \Delta_{(\alpha-2)/2}p_h,  w(x) \right)_{I_i} \\
\left(p_h, z(x) \right)_{I_i}  =\sqrt{\varepsilon}\left(\frac{\partial u_h}{\partial x},  z(x) \right)_{I_i}\\
\left(u_h(x,0),v(x)\right)_{I_i}=\left(u_0(x),v(x)\right)_{I_i}
\end{array}\right.
\end{eqnarray}
To complete the LDG scheme,  we introduce some notations and the numerical flux. Define  $$u^{\pm}(x_j) = \lim_{x\to x^{\pm}_j}u(x), \;\; 
\bar{u}_i=\frac{u^++u^-}{2},\quad \llbracket u\rrbracket = u^+-u^-$$
and the numerical flux as, 
$$\hat{u} = h_u(u^-,u^+),\quad \hat{q} = h_q(q^-,q^+),\quad \hat{f_h} = \hat{f}(u_h^-,u_h^+)$$
For the high order derivative part, a good choice is the alternating direction flux\cite{LDG1,shu0}, defined as,
$$\hat{u}_{i+\frac{1}{2}} = u^-_{i+\frac{1}{2}},  \quad \hat{q}_{i+\frac{1}{2}}=q^+_{i+\frac{1}{2}},   0\leqslant i \leqslant K-1 $$ or $$ \hat{u}_i = u^+_{i+\frac{1}{2}}, \quad\hat{q}_i=q^-_{i+\frac{1}{2}},1\leqslant i \leqslant K$$ 
For the nonlinear part, $\hat{f}$, any monotone flux can be used \cite{Jan}.

Applying integration by parts to \eqref{fracvar},  and replacing the fluxes at the interfaces by the corresponding numerical fluxes, we obtain, 

\begin{eqnarray}
\label{sch1}\left((u_h)_t, v\right)_{I_i} + \left(\hat{f_h}v-\sqrt{\varepsilon}\hat{q_h}v\right)|^{x^-_{i+\frac{1}{2}}}_{x^+_{i-\frac{1}{2}}} -\left(f(u_h)-\sqrt{\varepsilon}q_h,v_x\right)_{I_i} & =& 0\\
\label{sch2}\left(q_h, w(x) \right)_{I_i} - \left( \Delta_{(\alpha-2)/2}p_h,  w(x) \right)_{I_i}& = & 0 \\
\label{sch3}\left(p_h, z(x) \right)_{I_i}  - \sqrt{\varepsilon}\hat{u_h}z|^{x^-_{i+\frac{1}{2}}}_{x^+_{i-\frac{1}{2}}} +\sqrt{\varepsilon}\left(u_h,  z_x \right)_{I_i}&=&0\\
\label{sch4}\left(u_h(x,0),v(x)\right)_{I_i}-\left(u_0(x),v(x)\right)_{I_i}&=&0
\end{eqnarray}

\begin{remark}Originally, the problem is defined in $\mathbb{R}$. However, for numerical purposes, we assume there existing a domain $\Omega=[a,b]\subset \mathbb{R}$ in which $u$\ has compact support and restrict the problem to this domain $\Omega$. As a consequence, we impose homogeneous Dirichlet boundary conditions for all $u\in \mathbb{R}\backslash\Omega$. to recover
\begin{eqnarray}
-(-\Delta)^{\alpha/2}u(x) =-\frac{{_{-\infty}D_x^\alpha}u(x) + {_xD_{\infty}^\alpha}u(x)}{2\cos\left(\frac{\alpha\pi}{2}\right)} = -\frac{{_{a}D_x^\alpha}u(x) + {_xD_{b}^\alpha}u(x)}{2\cos\left(\frac{\alpha\pi}{2}\right)}
\end{eqnarray}
\end{remark}
For the flux at the boundary, we use the flux introduced in \cite{castillo}, defined as
$$\hat{u}_{K+\frac{1}{2}}=u(b,t),\quad \hat{q}_{K+\frac{1}{2}}=q^-_{K+\frac{1}{2}}+\frac{\beta}{h}\llbracket u_{K+\frac{1}{2}}\rrbracket$$
or $$\hat{u}_{\frac{1}{2}}=u(a,t),\quad \hat{q}_{\frac{1}{2}}=q^-_{\frac{1}{2}}+\frac{\beta}{h}\llbracket u_{\frac{1}{2}}\rrbracket$$
where $\beta$ is a positive constant.

\section{Stability and error estimates}
In the following we shall discuss stability and accuracy of the proposed scheme, both for the fractional diffusion problem and the more general convection-diffusion problem.

\subsection{Stability}
In order to carry on analysis of the LDG scheme, we define
\begin{eqnarray}\label{bilinear}
\nonumber \mathscr{B}(u,p,q; v, w, z) &=&\int_{0}^T\sum_{i=1}^K \left(u_t, v\right)_{I_i}dt + \int_{0}^T\sum_{i=1}^K\left(\hat{f}v-\sqrt{\varepsilon}\hat{q}v\right)|^{x^-_{i+\frac{1}{2}}}_{x^+_{i-\frac{1}{2}}}dt   \\
&&-\int_{0}^T\sum_{i=1}^K\left(f(u)-\sqrt{\varepsilon}q,v_x\right)_{I_i}dt+\int_{0}^T\sum_{i=1}^K\left(q, w(x) \right)_{I_i}dt \\
\nonumber&&- \int_{0}^T\sum_{i=1}^K\left(\Delta_{(\alpha-2)/2}p,  w(x) \right)_{I_i}dt- \int_{0}^T\sum_{i=1}^K\sqrt{\varepsilon}\hat{u}z|^{x^-_{i+\frac{1}{2}}}_{x^+_{i-\frac{1}{2}}}dt \\
 \nonumber && +\int_{0}^T\sum_{i=1}^K\left(p, z(x) \right)_{I_i}dt+\int_{0}^T\sum_{i=1}^K\sqrt{\varepsilon}\left(u,  z_x \right)_{I_i} dt-\int_{0}^T\mathscr{L}(v, w, z)dt
\end{eqnarray}
where $\mathscr{L}$ contains the boundary term, defined as
\begin{equation}
\mathscr{L}(v, w, z) =   \sqrt{\varepsilon} u(a,t) z^+_\frac{1}{2}- \frac{\sqrt{\varepsilon}\beta}{h}  u(b,t) v^-_{K+\frac{1}{2}} dt +\sqrt{\varepsilon} u(b,t)z^-_{K+\frac{1}{2}} =0
\end{equation}
If $(u,p,q)$ is a solution, then $\mathscr{B}(u,p,q; v, w, z) =0$ for any $ (v,w,z)$.
Considering the fluxes $\hat{u}_{i+\hf}=u^-_{i+\hf},\hat{q}_{i+\hf}=q^+_{i+\hf}$ and the flux at the boundaries we recover,
\begin{eqnarray}\label{B2}
\nonumber \mathscr{B}(u,p,q; v, w, z) &=& \int_{0}^T\sum_{i=1}^K \left(u_t, v\right)_{I_i} dt-\int_{0}^T\sum_{i=1}^K\left(f(u),v_x\right)_{I_i}dt \\
\nonumber &&+\int_{0}^T\sum_{i=1}^K\left(\sqrt{\varepsilon}q,v_x\right)_{I_i}dt +\int_{0}^T\sum_{i=1}^K\sqrt{\varepsilon}\left(u,  z_x \right)_{I_i} dt +\int_{0}^T\sum_{i=1}^K\left(q, w(x) \right)_{I_i}dt\\
\nonumber && - \int_{0}^T\sum_{i=1}^K\left(\Delta_{(\alpha-2)/2}p,  w(x) \right)_{I_i} dt+\int_{0}^T\sum_{i=1}^K\left(p, z(x) \right)_{I_i}dt \\
&&-\int_{0}^T\sum_{i=1}^{K-1}{\hat{f}_{i+\frac{1}{2}}}\llbracket v\rrbracket_{i+\frac{1}{2}} dt +\int_{0}^T\sum_{i=1}^{K-1}\sqrt{\varepsilon}{q^+_{i+\frac{1}{2}}}\llbracket v\rrbracket_{i+\frac{1}{2}} dt \\
\nonumber &&+\int_{0}^T\sum_{i=1}^{K-1}\sqrt{\varepsilon}{u^-_{i+\frac{1}{2}}}\llbracket z\rrbracket_{i+\frac{1}{2}} dt -\int_{0}^T (\hat{f}_{\frac{1}{2}}v^+_\frac{1}{2}-\hat{f}_{K+\frac{1}{2}}v^-_{K+\frac{1}{2}}) dt\\
\nonumber&& +\int_{0}^T  \sqrt{\varepsilon} (q^+_{\frac{1}{2}}v^+_\frac{1}{2}-  q^-_{K+\frac{1}{2}}v^-_{K+\frac{1}{2}})dt  +\int_{0}^T \frac{\sqrt{\varepsilon}\beta}{h} u^-_{K+\frac{1}{2}}v^-_{K+\frac{1}{2}} dt 
\end{eqnarray}

\begin{lemma}\label{lem41} Set $(v,w,z)=(u,-p,q)$ in \eqref{B2}, and define $\Phi(u)=\int^u f(u)du$. Then the following result holds,
\begin{eqnarray*}
\mathscr{B}(u,p,q; u, -p, q) &=& \|u(x,T)\|^2-\|u_0\|^2+\int_{0}^T (\Delta_{(\alpha-2)/2}p,p)dt + \int_{0}^T \frac{\sqrt{\varepsilon}\beta}{h} (u^-_{K+\frac{1}{2}})^2 dt \\
&&+ \int_{0}^T\Phi(u)_{\frac{1}{2}}- \Phi(u)_{K+\frac{1}{2}}- (\hat{f}u)_{\frac{1}{2}}+(\hat{f}u)_{K+\frac{1}{2}}dt\\
&&+\int_{0}^T\sum_{j=1}^{K-1} \left(\llbracket\Phi(u)\rrbracket_{j+\frac{1}{2}} - \hat{f}\llbracket u\rrbracket_{j+\frac{1}{2}}\right)dt 
\end{eqnarray*}

\end{lemma}

\begin{proof}
Set $(v,w,z)=(u,-p,q)$ in \eqref{B2}, and consider the integration by parts formula $(q,u_x)_{I_i}+(u,q_x)_{I_i}=(uq)|_{x^+_{i-\hf}}^{x^-_{i+\hf}}$, to recover at an interface
\begin{eqnarray*}
& &\sum_{i=1}^K\left(\sqrt{\varepsilon}q,v_x\right)_{I_i} +\sum_{i=1}^K\left(\sqrt{\varepsilon}u,  z_x \right)_{I_i}+\sum_{i=1}^{K-1}\sqrt{\varepsilon}{q^+_{i+\frac{1}{2}}}\llbracket v\rrbracket_{i+\frac{1}{2}} 
+\sum_{i=1}^{K-1}\sqrt{\varepsilon}{u^-_{i+\frac{1}{2}}}\llbracket z\rrbracket_{i+\frac{1}{2}}\\
&=& \sum_{i=1}^K\sqrt{\varepsilon}\left(uq\right)|^{x^-_{i+\hf}}_{x^+_{i-\hf}}+\sum_{i=1}^{K-1}\sqrt{\varepsilon}{q^+_{i+\frac{1}{2}}}\llbracket u\rrbracket_{i+\frac{1}{2}} 
+\sum_{i=1}^{K-1}\sqrt{\varepsilon}{u^-_{i+\frac{1}{2}}}\llbracket q\rrbracket_{i+\frac{1}{2}}\\
&=& -\sqrt{\varepsilon} q^+_{\frac{1}{2}}u^+_\frac{1}{2}+ \sqrt{\varepsilon}  q^-_{K+\frac{1}{2}}u^-_{K+\frac{1}{2}}
\end{eqnarray*}
Then
\begin{eqnarray}\label{Bdiff}
\nonumber \mathscr{B}(u,p,q; u, -p, q) &=& \int_{0}^T\sum_{i=1}^K \left(u_t, u\right)_{I_i} dt-\int_{0}^T\sum_{i=1}^K\left(f(u),u_x\right)_{I_i}dt \\
&& - \int_{0}^T\sum_{i=1}^K\left(\Delta_{(\alpha-2)/2}p,  p \right)_{I_i} dt-\int_{0}^T\sum_{i=1}^{K-1}{\hat{f}_{i+\frac{1}{2}}}\llbracket v\rrbracket_{i+\frac{1}{2}} dt\\
\nonumber&& -\int_{0}^T (\hat{f}_{\frac{1}{2}}u^+_\frac{1}{2}-\hat{f}_{K+\frac{1}{2}}u^-_{K+\frac{1}{2}}) dt +\int_{0}^T \frac{\sqrt{\varepsilon}\beta}{h} \left(u^-_{K+\frac{1}{2}}\right)^2 dt 
\end{eqnarray}
Define $\Phi(u)=\int^uf(u)du$, then 
\begin{eqnarray}\label{phiu}
\sum_{i=1}^K\left(f(u),u_x\right)_{I_i}=\sum_{i=1}^K\Phi(x)|_{x^+_{i-\hf}}^{x^-_{i+\hf}} = -\sum_{i=1}^{K-1}\llbracket \Phi(u)\rrbracket_{i+\frac{1}{2}} -\Phi(u)_{\frac{1}{2}}+\Phi(u)_{K+\frac{1}{2}}
\end{eqnarray}
Combining  \eqref{Bdiff} and \eqref{phiu} proves the Lemma.
\end{proof} \\

\begin{theorem}\label{stability}
The semi-discrete scheme \eqref{sch1}-\eqref{sch4} is stable, and $\|u_h(x,T)\|\leqslant \|u_0(x)\|$ for any $T>0$. \\
\end{theorem}
\begin{proof}
Using the properties of the monotone flux, $\hat{f}(u^-,u^+)$ is a non-decreasing function of its first argument, and a non-increasing function of its second argument. Hence, we have $\llbracket\Phi(u_h)\rrbracket_{j+\frac{1}{2}} - \hat{f_h}\llbracket u_h\rrbracket_{j+\frac{1}{2}}>0, 1 \leqslant j \leqslant K-1$. By Galerkin orthogonality, $\mathscr{B}(u_h,p_h,q_h; u_h, -p_h, q_h)=0$, Lemma \ref{lem41} yields
\begin{eqnarray*}
&&\|u(x,T)\|^2-\|u_0\|^2+\int_{0}^T (\Delta_{(\alpha-2)/2}p,p)dt + \int_{0}^T \frac{\sqrt{\varepsilon}\beta}{h} (u^-_{K+\frac{1}{2}})^2 dt \\
&&+ \int_{0}^T\Phi(u)_{\frac{1}{2}}- \Phi(u)_{K+\frac{1}{2}}- (\hat{f}u)_{\frac{1}{2}}+(\hat{f}u)_{K+\frac{1}{2}}dt \leqslant 0
\end{eqnarray*}

Considering the boundary condition and Lemma \ref{fracnorm2}, we obtain $\|u_h(x,T)\|\leqslant \|u_0(x)\|$, hence completing the proof.
\end{proof}

\subsection{Error estimates} 

To estimate the error, we firstly consider fractional diffusion with the Laplacian operator, i.e., the case with $f=0$. For fractional diffusion,  \eqref{sch1}-\eqref{sch4} reduce to
\begin{eqnarray}
\label{dfsch1}\left((u_h)_t, v\right)_{I_i} -(\sqrt{\varepsilon}\hat{q_h}v)|^{x^-_{i+\frac{1}{2}}}_{x^+_{i-\frac{1}{2}}} +\left(\sqrt{\varepsilon}q_h,v_x\right)_{I_i} & =& 0\\
\label{dfsch2}\left(q_h, w(x) \right)_{I_i} - \left( \Delta_{(\alpha-2)/2}p_h,  w(x) \right)_{I_i}& = & 0 \\
\label{dfsch3}\left(p_h, z(x) \right)_{I_i}  - \sqrt{\varepsilon}\hat{u_h}z|^{x^-_{i+\frac{1}{2}}}_{x^+_{i-\frac{1}{2}}} +\sqrt{\varepsilon}\left(u_h,  z_x \right)_{I_i}&=&0\\
\label{dfsch4}\left(u_h(x,0),v(x)\right)_{I_i}-\left(u_0(x),v(x)\right)_{I_i}&=&0
\end{eqnarray}
Correspondingly, we have the compact form of the scheme as

\begin{eqnarray} 
\nonumber \mathscr{B}(u,p,q; v, w, z) &=& \int_{0}^T\sum_{i=1}^K \left(u_t, v\right)_{I_i} dt+\int_{0}^T\sum_{i=1}^K\sqrt{\varepsilon}\left(q,v_x\right)_{I_i}dt +\int_{0}^T\sum_{i=1}^K\sqrt{\varepsilon}\left(u,  z_x \right)_{I_i} dt \nonumber \\
&&+\int_{0}^T\sum_{i=1}^K\left(q, w(x) \right)_{I_i}dt - \int_{0}^T\sum_{i=1}^K\left(\Delta_{(\alpha-2)/2}p,  w(x) \right)_{I_i} dt
+\int_{0}^T\sum_{i=1}^K\left(p, z(x) \right)_{I_i}dt \nonumber \\
 \nonumber &&+\int_{0}^T\sum_{i=1}^{K-1}\sqrt{\varepsilon}{q^+_{i+\frac{1}{2}}}\llbracket v\rrbracket_{i+\frac{1}{2}} dt +\int_{0}^T\sum_{i=1}^{K-1}\sqrt{\varepsilon}{u^-_{i+\frac{1}{2}}}\llbracket z\rrbracket_{i+\frac{1}{2}} dt  \\
&& +\int_{0}^T  \sqrt{\varepsilon} q^+_{\frac{1}{2}}v^+_\frac{1}{2}dt+\int_{0}^T \frac{\sqrt{\varepsilon}\beta}{h} u^-_{K+\frac{1}{2}}v^-_{K+\frac{1}{2}}dt- \int_{0}^T \sqrt{\varepsilon} q^-_{K+\frac{1}{2}}v^-_{K+\frac{1}{2}}  dt \label{eqdiff}
\end{eqnarray}
To prepare for the main result, we first obtain a few central Lemmas. We define special projections, $\mathscr{P}^{\pm}$ and $\mathscr{Q}$ into $V_h$, which satisfy, for each $j$,
\begin{eqnarray}
&&\int_{I_j}(\mathscr{P}^{\pm}u(x)-u(x))v(x)dx = 0\quad \forall v\in P^{k-1}\quad \text{and} \quad \mathscr{P}^{\pm}u_{j+\frac{1}{2}} = u(x^\pm_{j+\frac{1}{2}})\\
&&\int_{I_j}(\mathscr{Q}u(x)-u(x))v(x)dx = 0\quad \forall v\in P^{k}
\end{eqnarray}
Denote $e_u=u-u_h,e_p=p-p_h,e_q=q-q_h$, then 
$$\mathscr{P}^-e_u=\mathscr{P}^-u-u_h,\mathscr{P}^+e_q=\mathscr{P}^+q-q_h,\mathscr{Q}e_p=\mathscr{Q}p-p_h$$
For any $(v,w,z)\in H^1(\Omega,\mathcal{T})\times L^2(\Omega,\mathcal{T}) \times L^2(\Omega,\mathcal{T})$ ,
\begin{equation}
\mathscr{B}(u,p,q; v, w, z) = \mathscr{L}(v, w, z) 
\end{equation}
Hence, $\mathscr{B}(e_u,e_p,e_q; v, w, z) =0$ and we recover 
\begin{eqnarray*}\label{Bequ}
&&\mathscr{B}(\mathscr{P}^- e_u,\mathscr{Q} e_p,\mathscr{P}^+ e_q;\mathscr{P}^- e_u,-\mathscr{Q} e_p,\mathscr{P}^+ e_q)\\
 &=& \mathscr{B}(\mathscr{P}^- e_u-e_u,\mathscr{Q} e_p-e_p,\mathscr{P}^+ e_q-e_q;\mathscr{P}^- e_u,-\mathscr{Q} e_p,\mathscr{P}^+ e_q) \\
&=&\mathscr{B}(\mathscr{P}^- u-u,\mathscr{Q}p-p,\mathscr{P}^+ q-q;\mathscr{P}^- e_u,-\mathscr{Q} e_p,\mathscr{P}^+ e_q) 
\end{eqnarray*}
Substitute  $(\mathscr{P}^-u-u, \mathscr{Q}p-p, \mathscr{P}^+q-q; \mathscr{P}^-e_u, -\mathscr{Q}e_p, \mathscr{P}^+e_q) $ into \eqref{eqdiff}, to recover the following Lemma,

\begin{lemma}\label{lem43}
For the bilinear form, \eqref{eqdiff}, we have
\begin{eqnarray*}
&&\mathscr{B}(\mathscr{P}^-u-u,\mathscr{Q}p-p,\mathscr{P}^+q-q; \mathscr{P}^- e_u,-\mathscr{Q} e_p,\mathscr{P}^+ e_q)\\
&\leqslant&\int_{0}^T\sum_{i=1}^K \left((\mathscr{P}^- u)_t-u_t, \mathscr{P}^- e_u\right)_{I_i} dt+C_{T,a,b}h^{2k+2}+\frac{1}{C_{T,a,b}}\int_{0}^T\sum_{i=1}^K\|\mathscr{Q}e_p\|_{I_i}^2dt\\
&&+\int_{0}^T\frac{\sqrt{\varepsilon}\beta}{h}|(\mathscr{P}^- e_u)^-_{K+\frac{1}{2}}|^2dt
\end{eqnarray*}
where $C_{T,a,b}$is independent of $h$, but may dependent of $T$ and $\Omega$.  
\end{lemma}
\\ \begin{proof}
From \eqref{eqdiff} we have
\begin{eqnarray*}
&&\mathscr{B}(\mathscr{P}^-u-u,\mathscr{Q}p-p,\mathscr{P}^+q-q; \mathscr{P}^- e_u,-\mathscr{Q} e_p,\mathscr{P}^+ e_q)\\
&=&  \int_{0}^T\sum_{i=1}^K \left((\mathscr{P}^- u)_t-u_t, \mathscr{P}^- e_u\right)_{I_i} dt 
       +\int_{0}^T\sum_{i=1}^K \sqrt{\varepsilon}\left(\mathscr{P}^+q-q, (\mathscr{P}^- e_u)_x\right)_{I_i} dt \\
&& + \int_{0}^T\sum_{i=1}^K\sqrt{\varepsilon}\left(\mathscr{P}^-u-u,(\mathscr{P}^+e_q)_x\right)_{I_i}dt                   
      - \int_{0}^T\sum_{i=1}^K \left(\mathscr{P}^+q-q,\mathscr{Q}e_p\right)_{I_i}dt  \\
&& + \int_{0}^T\sum_{i=1}^K\left(\Delta_{(\alpha-2)/2}(\mathscr{Q} p-p),  \mathscr{Q} e_p \right)_{I_i} dt
      + \int_{0}^T\sum_{i=1}^K \left(\mathscr{Q}p-p,\mathscr{P}^+e_q\right)_{I_i}dt \\
&&+ \int_{0}^T\sum_{i=1}^{K-1}\sqrt{\varepsilon} (\mathscr{P}^+q-q)^+_{i+\frac{1}{2}}\llbracket \mathscr{P}^-e_u \rrbracket_{i+\frac{1}{2}}   
      +\int_{0}^T\sum_{i=1}^{K-1}\sqrt{\varepsilon} (\mathscr{P}^-u-u)^-_{i+\frac{1}{2}}\llbracket \mathscr{P}^+e_q\rrbracket_{i+\frac{1}{2}}   \\
&& +\int_{0}^T \sqrt{\varepsilon} (\mathscr{P}^+q-q)^+_{\frac{1}{2}}\llbracket\mathscr{P}^- e_u^+\rrbracket_{\frac{1}{2}}  dt 
      +\int_{0}^T \frac{\sqrt{\varepsilon}\beta}{h} (\mathscr{P}^-u-u)^-_{K+\frac{1}{2}}\llbracket\mathscr{P}^- e_u^-\rrbracket_{K+\frac{1}{2}}  dt \\
&&-\int_{0}^T \sqrt{\varepsilon} (\mathscr{P}^+q-q)^-_{K+\frac{1}{2}}\llbracket\mathscr{P}^- e_u^-\rrbracket_{K+\frac{1}{2}}  dt 
\end{eqnarray*}
Since  $(\mathscr{P}^-e_u)_x\in P^{k-1},(\mathscr{P}^+e_q)_x\in P^{k-1},(\mathscr{Q}e_p)_x\in P^{k-1},\mathscr{Q}e_p\in P^{k}$, by the properties  of the projection $\mathscr{P}^\pm,\mathscr{Q}$, we recover,
\begin{eqnarray*}
&&\left(\mathscr{P}^+q-q, (\mathscr{P}^- e_u)_x\right)_{I_i} =0,\quad \left(\mathscr{P}^-u-u,(\mathscr{P}^+e_q)_x\right)_{I_i}=0\\
&&\left(\mathscr{Q}p-p,\mathscr{P}^+e_q\right)_{I_i}=0, \quad  \left(\mathscr{Q}p-p,(\mathscr{P}^+e_q)_x\right)_{I_i}=0\\
&&(\mathscr{P}^-u-u)_{i+\frac{1}{2}}=0,\quad (\mathscr{P}^+q-q)_{i+\hf} =0 .
\end{eqnarray*}
We obtain,
\begin{eqnarray*}
&&\mathscr{B}(\mathscr{P}^-u-u,\mathscr{Q}p-p,\mathscr{P}^+q-q; \mathscr{P}^- e_u,-\mathscr{Q} e_p,\mathscr{P}^+ e_q)\\
&=& \int_{0}^T\sum_{i=1}^K \left((\mathscr{P}^- u)_t-u_t, \mathscr{P}^- e_u\right)_{I_i} dt   -\int_{0}^T \sqrt{\varepsilon} (\mathscr{P}^+q-q^-)_{K+\frac{1}{2}}(\mathscr{P}^- e_u)^-_{K+\frac{1}{2}}  dt \\
&& +\int_{0}^T\sum_{i=1}^K\left(\Delta_{(\alpha-2)/2}(\mathscr{Q} p-p)-(\mathscr{P}^+q-q),  \mathscr{Q} e_p \right)_{I_i} dt    
\end{eqnarray*}
Recalling the projection property and Lemma \ref{approrder}, we obtain
$\|\Delta_{(\alpha-2)/2}(\mathscr{Q} p-p)-(\mathscr{P}^+q-q)\|\leqslant Ch^{k+1}$. Combining this with  Young's inequality and \eqref{inverse}, we obtain
\begin{eqnarray*}
&&\mathscr{B}(\mathscr{P}^-u-u,\mathscr{Q}p-p,\mathscr{P}^+q-q; \mathscr{P}^- e_u,-\mathscr{Q} e_p,\mathscr{P}^+ e_q)\\
&\leqslant& \int_{0}^T\sum_{i=1}^K \left((\mathscr{P}^- u)_t-u_t, \mathscr{P}^- e_u\right)_{I_i} dt +C_{T,a,b}h^{2k+2}+\frac{1}{C_{T,a,b}}\int_{0}^T\sum_{i=1}^K\| \mathscr{Q} e_p \|_{I_i}^2 dt   \\
&&+\int_{0}^T\frac{\sqrt{\varepsilon}\beta }{h}| (\mathscr{P}^- e_u)_{K+\frac{1}{2}}|^2dt 
\end{eqnarray*}
This proves the lemma.
\end{proof} \\

\begin{lemma}\cite{castillo}\label{evaltime}
Suppose that for all $t>0$ we have
$$\chi^2(t)+R(t)\leqslant A(t)+2\int_{0}^tB(s)\chi(s)ds,$$
where $R,A,B$ are nonnegative functions. Then, for any $T>0$,
$$\sqrt{\chi^2(T)+R(t)}\leqslant\sup_{0\leqslant t\leqslant T}A^{1/2}(t)+\int_{0}^TB(t)dt$$ \\
\end{lemma}

\begin{theorem}\label{thm1}
Let $u$ be the a sufficiently smooth exact solution to \eqref{problem} in $\Omega \subset \mathbb{R}$ with $f(u)=0$. Let $u_h$ be the numerical solution of the semi-discrete LDG scheme \eqref{sch1}-\eqref{sch4}, then for small enough $h$, we have the following error estimates
\begin{equation*}
\|u-u_h\| \leqslant Ch^{k+1}
\end{equation*}
where $C$ is a constant independent of $h$. \\
\end{theorem}
\begin{proof}
From Lemma \ref{lem41} and the initial error $\|\mathscr{P}^-e_u(0)\| =0$, we have
\begin{eqnarray*}
&&B(\mathscr{P}^- e_u,\mathscr{Q} e_p,\mathscr{P}^+ e_q;\mathscr{P}^- e_u,-\mathscr{Q} e_p,\mathscr{P}^+ e_q) \\
&=&\frac{1}{2}\|\mathscr{P}^-e_u(T)\|^2  +\int_{0}^T\left(\Delta_{(\alpha-2)/2}\mathscr{Q} e_p,  \mathscr{Q} e_p \right)_{I_i} dt+\int_{0}^T \frac{\sqrt{\varepsilon}\beta}{h} |\mathscr{P}^-e_u|^2_{K+\frac{1}{2}} dt 
\end{eqnarray*}
Combining this with Lemma \ref{lem43} and \eqref{Bequ}, we have
\begin{eqnarray*}
&&\frac{1}{2}\|\mathscr{P}^-e_u(T)\|^2 +\int_{0}^T\left(\Delta_{(\alpha-2)/2}\mathscr{Q} e_p,  \mathscr{Q} e_p \right)dt \\
&\leqslant& \int_{0}^T \left((\mathscr{P}^- u)_t-u_t, \mathscr{P}^- e_u\right)dt +C_{T,a,b}h^{2k+2}+\frac{1}{C_{T,a,b}}\int_{0}^T \| \mathscr{Q} e_p \|^2 dt  
\end{eqnarray*}
Recalling  the fractional Poincar\'e-Friedrichs Lemma \ref{fracpc}, we get
\begin{eqnarray*}
&&\frac{1}{2}\|\mathscr{P}^-e_u(T)\|^2 \leqslant \int_{0}^T \left((\mathscr{P}^- u)_t-u_t, \mathscr{P}^- e_u\right)dt +C_{T,a,b}h^{2k+2}
\end{eqnarray*}
Using Lemma \ref{evaltime} and the error associated with the projection error, proves the theorem.
\end{proof}

For the more general fractional convection-diffusion problem, we introduce a few results and then give the error estimate. \\
 
\begin{lemma}\cite{shu1}
For any piecewise smooth function $\omega \in L^2(\Omega)$, on each cell boundary point we define
\begin{equation}
\kappa (\hat{f};\omega) \equiv \kappa(\hat{f};\omega^-,\omega^+)  \triangleq  \left\{\begin{array}{l l}
[\omega]^{-1}(f(\bar{\omega})-\hat{f}(\omega)) & \text{if} \quad [\omega]\neq 0,\\
\frac{1}{2}|f'(\bar{\omega})| & \text{if} \quad [\omega] = 0
\end{array}
\right.
\end{equation}
where $\hat{f}(\omega)\equiv \hat{f}(\omega^{-},\omega^+)$ is a monotone numerical flux consistent with the given flux $f$. Then $\kappa(\hat{f},\omega)$ is non-negative and bounded for any $(\omega^-,\omega^+) \in \mathbb{R}$. Moreover, we have 
\begin{eqnarray}
\frac{1}{2}|f'(\bar{\omega})| \leqslant  \kappa(\hat{f};\omega)+C_*|[\omega]|,\\
-\frac{1}{8}f''(\bar{\omega})[\omega] \leqslant \kappa(\hat{f};\omega)+C_*|[\omega]|^2.
\end{eqnarray}
\end{lemma}
Define 
\begin{eqnarray*}
\sum_{j=1}^K \mathscr{H}_j(f;u,u_h,v) &=& \sum_{j=1}^K\int_{I_j}(f(u)-f(u_h))v_x dx +\sum_{j=1}^K((f(u)-f(\bar{u_h}))[v])_{j+\frac{1}{2}}\\
&&+\sum_{j=1}^K((f(\bar{u_h})-\hat{f})[v])_{j+\frac{1}{2}}
\end{eqnarray*}
\begin{lemma}\cite{shu2}\label{nonlinear}
For $\mathscr{H}(f;u,u_h,v) $ defined above,  we have the following estimate
\begin{eqnarray*}
\sum_{j=1}^K \mathscr{H}_j(f;u,u_h,v) & \leqslant & -\frac{1}{4}\kappa(\hat{f};u_h)[v]^2+(C+C_*(\|v\|_\infty+h^{-1}\|e_u\|_\infty^2))\|v\|^2\\
&&+(C+C_*h^{-1}\|e_u\|_\infty^2)h^{2k+1}
\end{eqnarray*}
\end{lemma}
Combining Theorem \ref{thm1} and Lemma \ref{nonlinear}, we recover the following error estimate \\

\begin{theorem}
Let $u$ be the exact solution of \eqref{problem}, which is sufficiently smooth  $\in \Omega \subset \mathbb{R}$ and assume $f \in C^3$. Let $u_h$ be the numerical solution of the semi-discrete LDG scheme \eqref{dfsch1}-\eqref{dfsch4} and denote the corresponding numerical error by $e_u = u-u_h$. $V_h$ is  the space of piecewise polynomials of degree $k\geqslant 1$, then for small enough $h$, we have the following error estimates
\begin{equation*}
\|u-u_h\| \leqslant Ch^{k+\frac{1}{2}}
\end{equation*}

\end{theorem} 

\begin{remark} Although the order of convergence $k+\frac{1}{2}$ is obtained it can be improved if an upwind flux is chosen for $\hat{f}$. In this case,  optimal order can be recovered. We shall demonstrate this in the numerical examples. \\
\end{remark}

\begin{remark} For problems with mixed boundary conditions, another scheme may be more suitable for implementation. Suppose the problem has a Dirichlet boundary on the left end and a Neumann boundary at the right end. Let $p = \sqrt{\varepsilon}\frac{\partial u}{\partial x}, q=\sqrt{\varepsilon}\frac{\partial p}{\partial x}$, then we have
\begin{eqnarray*}
\left\{\begin{array}{l}
\frac{\partial u}{\partial t} + \frac{\partial}{\partial x}f(u) = \Delta_{(\alpha-2)/2}q\\
q =\sqrt{\varepsilon}\frac{\partial p}{\partial x} \\
p =\sqrt{\varepsilon}\frac{\partial u}{\partial x}\\
\end{array}\right.
\end{eqnarray*}
Similarly, we have the following scheme,
\begin{eqnarray}
&&\label{second1}\left((u_h)_t, v\right)_{I_i} +  \hat{f_h} v|_{I_i} -\left(f(u_h),v_x\right)_{I_i} -  \left(\Delta_{(\alpha-2)/2}q_h,  v(x) \right)_{I_i} = 0\\
&&\label{second2}\left(q_h, w(x) \right)_{I_i} - \sqrt{\varepsilon}\hat{p_h}w|_{I_i} +\sqrt{\varepsilon}\left(p_h,  w_x \right)_{I_i} =  0 \\
&&\label{second3}\left(p_h, z(x) \right)_{I_i}  - \sqrt{\varepsilon}\hat{u_h}z|_{I_i} +\sqrt{\varepsilon}\left(u_h,  z_x \right)_{I_i}=0\\
&&\label{second4}\left(u_h(x,0),v(x)\right)_{I_i}-\left(u_0(x),v(x)\right)_{I_i}=0
\end{eqnarray}
In this scheme, a mixed boundary condition is imposed naturally. However, the analysis is more complicated. While computational results indicate excellent behavior and optimal convergence, we shall not talk about the theoretical aspect of this scheme. 
\end{remark}

\section{Numerical examples}
In the following, we shall present a few results to numerically validate the analysis. 

The discussion so far have focused on the treatment of the spatial dimension with a semi-discrete form as
\begin{equation}
\frac{d\mathbf{u}_h}{dt} = \mathcal{L}_h(\mathbf{u}_h,t)
\end{equation}
where $\mathbf{u}_h$ is the vector of unknowns.
For the time discretization of the equations, we use a fourth order low storage explicit Runge-Kutta(LSERK) method\cite{Jan} of the form,
\begin{eqnarray*}
&& \mathbf{p}^{0} = \mathbf{u}^n,\\
&& i\in[1,\cdots,5]:\left\{\begin{array}{l}
\mathbf{k}^{i}=a_i\mathbf{k}^{i-1}+\Delta t\mathcal{L}_h(\mathbf{p}^{i-1},t^n+c_i\Delta t)\\
\mathbf{p}^{(i)}=\mathbf{p}^{(i-1)}+b_i\mathbf{k}^{(i)},\end{array}\right.\\
&&\mathbf{u}_h^{n+1}=\mathbf{p}^{(5)}
\end{eqnarray*}
where $ a_i,b_i,c_i$ are coefficients of LSRK method given in \cite{Jan}. For the explicit methods, proper timestep $\Delta t$ is needed. In our examples, numerically the condition $\Delta t\leqslant C \Delta x_{min}^\alpha,( 0<C<1)$ is necessary for stability. \\
\begin{table}[!ht]
\caption{Error and order of convergence for solving the fractional Laplacian problem with $K$ elements and polynomial order $N$.}
{\small 
\begin{center}
\begin{tabular}{r|| ccccccc|} \hline

&   \multicolumn{7}{c|}{$\alpha = 1.1$} \\ \hline
  K                               &             10        &       \multicolumn{2}{c}{20}
    &    \multicolumn{2}{c}{30}    &     \multicolumn{2}{c|}{40}  \\    
  
  N                              & $\|e_u \|_2$ & $\|e_u \|_2$ & Order &
$\|e_u \|_2$ & Order & $\|e_u \|_2$ & Order  \\ \hline
       & K=      10 & K=      20 &      order & K=      30 &      order &
K=      40 &      order \\ \hline
1 &   1.64e-06 &   3.98e-07 &       2.04 &   1.73e-07 &       2.05 &   9.60e-08
&       2.05 \\
2 &   1.24e-07 &   1.69e-08 &       2.87 &   5.13e-09 &       2.94 &   2.18e-09
&       2.97 \\
3 &   1.14e-08 &   7.23e-10 &       3.98 &   1.41e-10 &       4.03 &   4.44e-11
&       4.03 \\
 \hline
 \hline

&   \multicolumn{7}{c|}{$\alpha = 1.3$} \\ \hline
  K                               &             10        &       \multicolumn{2}{c}{20}
    &    \multicolumn{2}{c}{30}    &     \multicolumn{2}{c|}{40}  \\    
  
  N                              & $\|e_u \|_2$ & $\|e_u \|_2$ & Order &
$\|e_u \|_2$ & Order & $\|e_u \|_2$ & Order  \\ \hline
       & K=      10 & K=      20 &      order & K=      30 &      order &
K=      40 &      order \\ \hline
1 &   1.45e-06 &   3.46e-07 &       2.07 &   1.50e-07 &       2.06 &   8.33e-08
&       2.05 \\
2 &   1.19e-07 &   1.55e-08 &       2.94 &   4.64e-09 &       2.98 &   1.96e-09
&       2.99 \\
3 &   1.04e-08 &   6.48e-10 &       4.00 &   1.26e-10 &       4.03 &   3.97e-11
&       4.02 \\
 \hline
 \hline

&   \multicolumn{7}{c|}{$\alpha = 1.5$} \\ \hline
  K                               &             10        &       \multicolumn{2}{c}{20}
    &    \multicolumn{2}{c}{30}    &     \multicolumn{2}{c|}{40}  \\    
  
  N                              & $\|e_u \|_2$ & $\|e_u \|_2$ & Order &
$\|e_u \|_2$ & Order & $\|e_u \|_2$ & Order  \\ \hline
1 &   1.35e-06 &   3.24e-07 &       2.06 &   1.42e-07 &       2.04 &   7.90e-08
&       2.03 \\
2 &   1.22e-07 &   1.51e-08 &       3.01 &   4.51e-09 &       2.99 &   1.90e-09
&       2.99 \\
3 &   1.04e-08 &   6.28e-10 &       4.05 &   1.22e-10 &       4.04 &   3.85e-11
&       4.02 \\
 \hline
 \hline
 &   \multicolumn{7}{c|}{$\alpha = 1.8$} \\ \hline
  K                               &             10        &       \multicolumn{2}{c}{20}
    &    \multicolumn{2}{c}{30}    &     \multicolumn{2}{c|}{40}  \\    
  
  N                              & $\|e_u \|_2$ & $\|e_u \|_2$ & Order &
$\|e_u \|_2$ & Order & $\|e_u \|_2$ & Order  \\ \hline
1 &   1.28e-06 &   3.11e-07 &       2.04 &   1.38e-07 &       2.01 &   7.74e-08
&       2.01 \\
2 &   1.40e-07 &   1.51e-08 &       3.21 &   4.48e-09 &       3.00 &   1.89e-09
&       3.00 \\
3 &   1.44e-08 &   6.72e-10 &       4.42 &   1.23e-10 &       4.19 &   3.83e-11
&       4.05 \\
 \hline
 \hline

\end{tabular}
\end{center}}
\label{example1}
\end{table}%

{\bf{Example 1.}}
 As the first example, we consider the fractional diffusion equation with the fractional Laplacian, 
\begin{eqnarray}
\left\{\begin{array}{ll}
\frac{\partial u(x,t)}{\partial t}  =  -(-\Delta)^{\alpha/2} u(x,t)+g(x,t),\quad in\quad [0, 1]\times (0, 0.5]\\
u(x,0)=u_0(x).
\end{array}\right.
\end{eqnarray}
with the initial condition $u_0(x) = x^6(1-x)^6$, and the source term $g(x,t) = e^{-t}\left( -u_0(x)+ (-\Delta)^{\frac{\alpha}{2}}u_0(x)\right) $\\
The exact solution is $u(x,t) = e^{-t} x^6(1-x)^6$.

The results are shown in Table \ref{example1} and we observe optimal ${\cal O}(h^{k+1})$ order of convergence across $1<\alpha < 2$. \\

\begin{table}[!ht]
\caption{Error and order of convergence for solving the fractional Burgers'
equation with $K$ elements and polynomial order $N$.}
{\small 
\begin{center}
\begin{tabular}{r|| ccccccc|} \hline

&   \multicolumn{7}{c|}{$\alpha = 1.01$} \\ \hline
  K                               &             10        &       \multicolumn{2}{c}{20}
    &    \multicolumn{2}{c}{30}    &     \multicolumn{2}{c|}{40}  \\    
  
  N                              & $\|e_u \|_2$ & $\|e_u \|_2$ & Order &
$\|e_u \|_2$ & Order & $\|e_u \|_2$ & Order  \\ \hline
       & K=      10 & K=      20 &      order & K=      30 &      order &
K=      40 &      order \\ \hline
1 &   1.10e-03 &   2.81e-04 &       1.97 &   1.24e-04 &       2.03 &   6.90e-05
&       2.03 \\
2 &   6.53e-05 &   1.00e-05 &       2.70 &   3.09e-06 &       2.90 &   1.33e-06
&       2.94 \\
3 &   5.94e-06 &   4.00e-07 &       3.89 &   8.05e-08 &       3.96 &   2.58e-08
&       3.95 \\
 \hline
 \hline

&   \multicolumn{7}{c|}{$\alpha = 1.5$} \\ \hline
  K                               &             10        &       \multicolumn{2}{c}{20}
    &    \multicolumn{2}{c}{30}    &     \multicolumn{2}{c|}{40}  \\    
  
  N                              & $\|e_u \|_2$ & $\|e_u \|_2$ & Order &
$\|e_u \|_2$ & Order & $\|e_u \|_2$ & Order  \\ \hline
1 &   8.89e-04 &   2.15e-04 &       2.05 &   9.45e-05 &       2.03 &   5.28e-05
&       2.02 \\
2 &   6.71e-05 &   8.62e-06 &       2.96 &   2.57e-06 &       2.99 &   1.09e-06
&       2.99 \\
3 &   4.91e-06 &   3.34e-07 &       3.88 &   6.80e-08 &       3.93 &   2.16e-08
&       3.99 \\
 \hline
 \hline
 &   \multicolumn{7}{c|}{$\alpha = 1.8$} \\ \hline
  K                               &             10        &       \multicolumn{2}{c}{20}
    &    \multicolumn{2}{c}{30}    &     \multicolumn{2}{c|}{40}  \\    
  
  N                              & $\|e_u \|_2$ & $\|e_u \|_2$ & Order &
$\|e_u \|_2$ & Order & $\|e_u \|_2$ & Order  \\ \hline
1 &   8.43e-04 &   2.09e-04 &       2.01 &   9.25e-05 &       2.01 &   5.20e-05
&       2.00 \\
2 &   6.78e-05 &   8.59e-06 &       2.98 &   2.56e-06 &       2.99 &   1.08e-06
&       2.99 \\
3 &   4.80e-06 &   3.32e-07 &       3.85 &   6.84e-08 &       3.90 &   2.20e-08
&       3.94 \\
 \hline
 \hline

\end{tabular}
\end{center}}
\label{example2}
\end{table}%
{\bf{Example 2.}} We consider the fractional Burgers' equation
\begin{eqnarray}\label{ex2}
\left\{\begin{array}{ll}
\frac{\partial u(x,t)}{\partial t} + \frac{\partial} {\partial x}\left(\frac{u^2(x,t)}{2}\right)  =  \varepsilon(-(-\Delta)^{\alpha/2}) u(x,t)+g(x,t),\quad in\quad [-2, 2]\times (0, 0.5]\\
u(x,0)=u_0(x).
\end{array}\right.
\end{eqnarray}
with initial the condition 
$$u_0(x) = \left\{\begin{array}{ll}
(1-x^2)^4/10 &\ -1\leqslant x \leqslant 1\\
0 &\ \mbox{otherwise} \end{array}\right.$$
The source term $g(x,t) = e^{-t}\left( -u_0(x)+e^{-t}u_0(x)u'_0(x)+\varepsilon (-\Delta)^{\frac{\alpha}{2}}u_0(x)\right), \varepsilon =1 $ is derived to yield an exact solution as 
$$u(x,t) = \left\{
\begin{array}{ll}
e^{-t}(1-x^2)^4/10 &\ -1\leqslant x \leqslant 1\\
0 &\ \mbox{otherwise} \end{array}\right.$$
The results are shown in Table \ref{example2} and we observe optimal ${\cal
O}(h^{k+1})$ order of convergence across $1<\alpha < 2$. \\

{\bf Example 3} Let us finally consider the fractional Burgers' equation with a discontinuous initial condition,
$$u_0(x) = \left\{\begin{array}{ll}\frac{1}{2} & -1\leqslant x \leqslant 0\\0 & \mbox{otherwise} \end{array}\right.$$
We consider the case in  \eqref{ex2} with $\varepsilon =0.04$, $g(x,t) = 0$.
\noindent\begin{figure}[htbp] 
\begin{minipage}[b]{0.45\linewidth}
\centering
\includegraphics[width=1\linewidth]{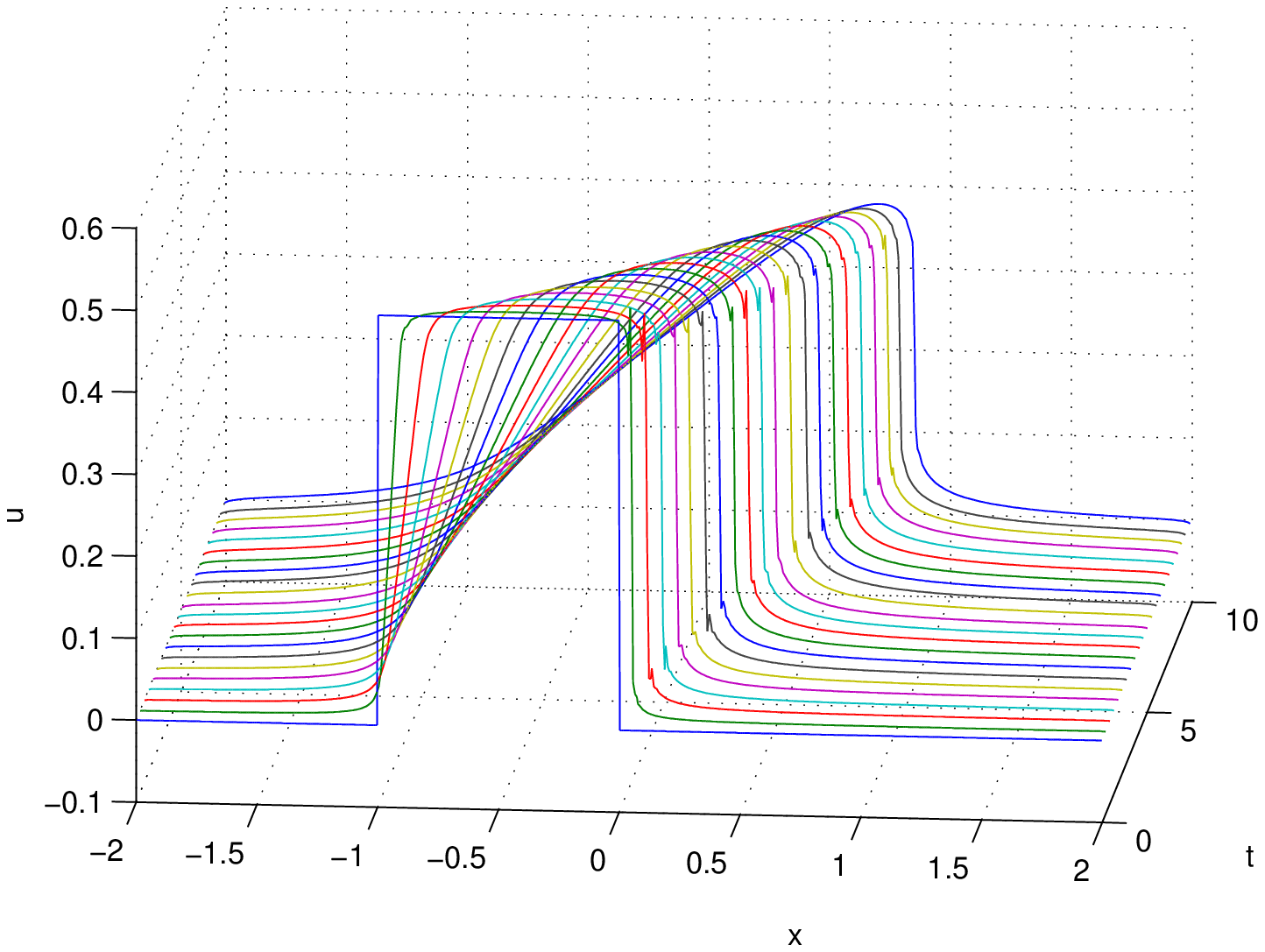}
\caption{Solution of the fractional Burgers's equation with $\epsilon=0.04,\alpha=1.005$}
\label{fig31}
\end{minipage}
\begin{minipage}[b]{0.45\linewidth}
\centering
\includegraphics[width=1\linewidth]{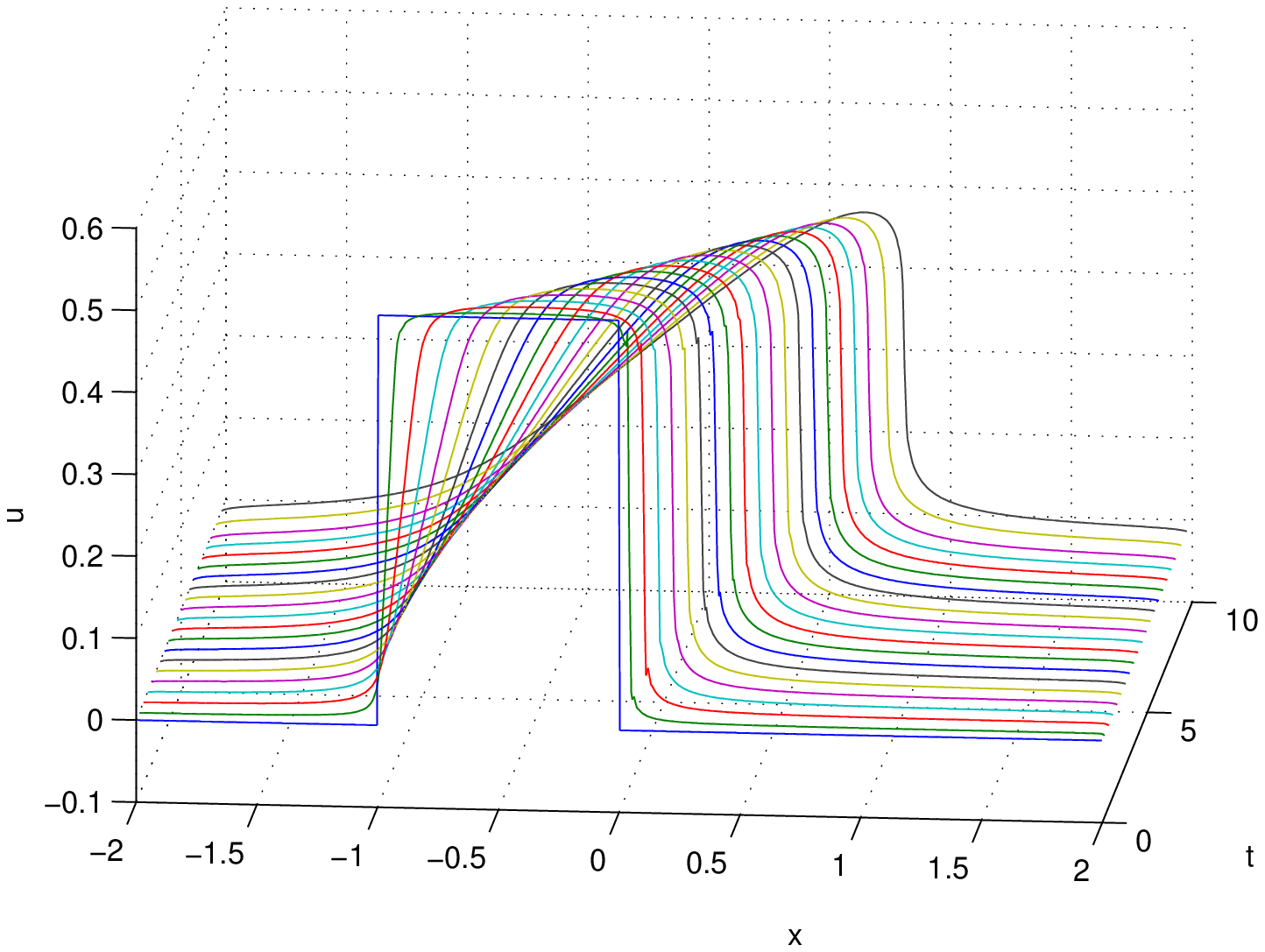} 
\caption{Solution of the fractional Burgers's equation with $\epsilon=0.04,\alpha=1.1$}
\label{fig32}
\end{minipage}
\end{figure}
\noindent\begin{figure}[htbp] 
\begin{minipage}[b]{0.45\linewidth}
\centering
\includegraphics[width=1\linewidth]{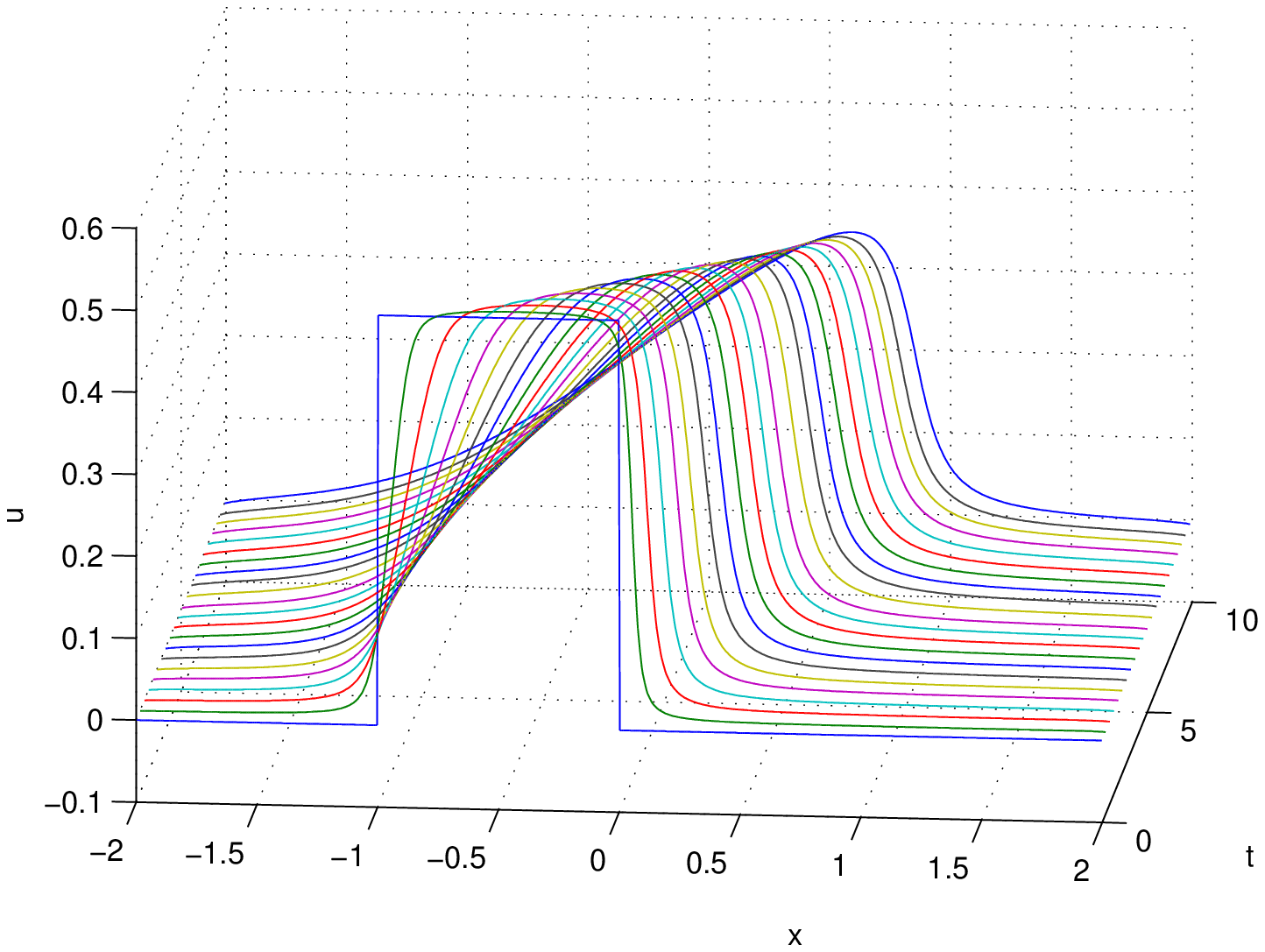}
\caption{Solution of the fractional Burgers's equation with $\epsilon=0.04,\alpha=1.5$}
\label{fig33}
\end{minipage}
\begin{minipage}[b]{0.45\linewidth}
\centering
\includegraphics[width=1\linewidth]{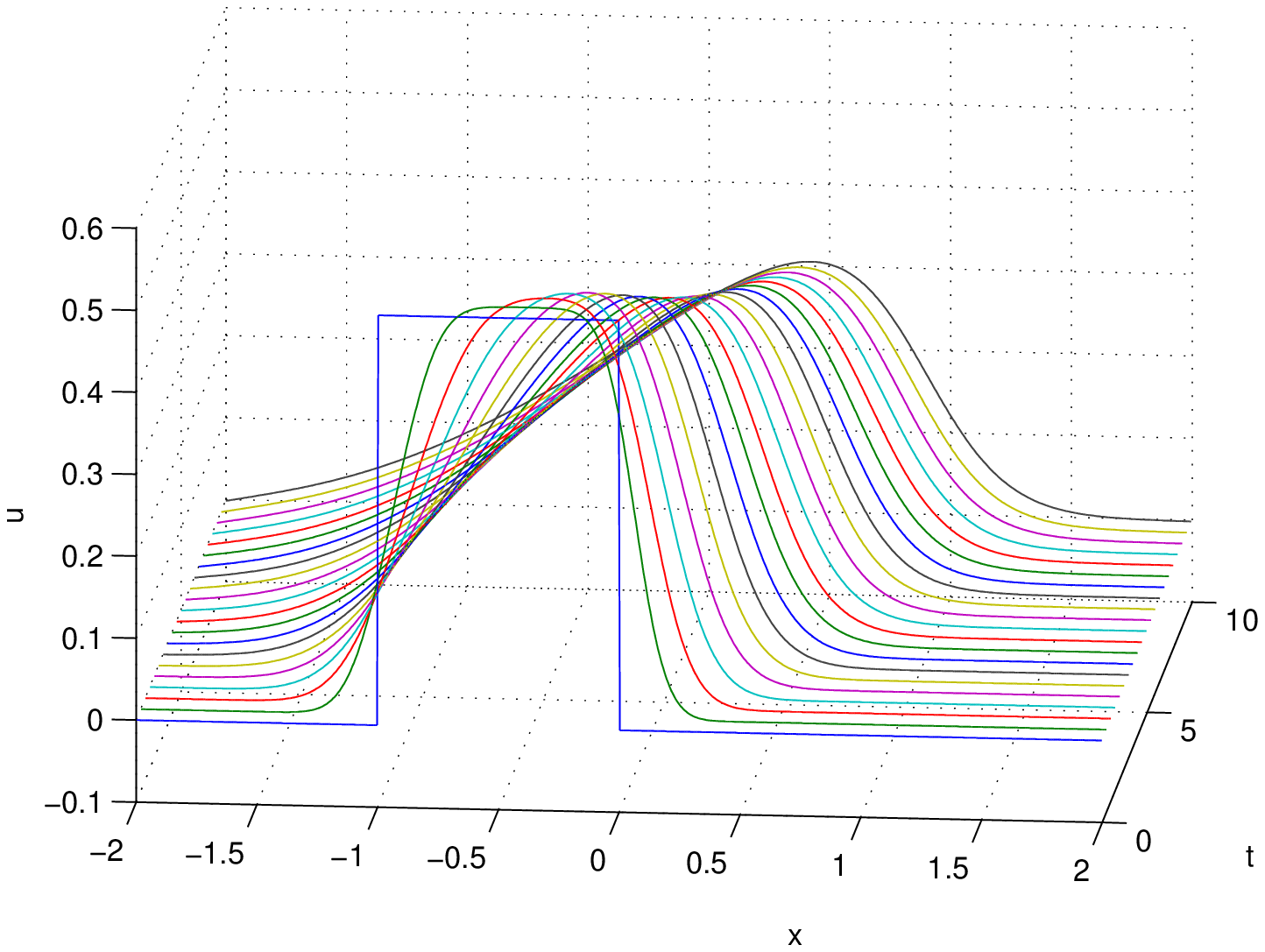} 
\caption{Solution of the fractional Burgers's equation with $\epsilon=0.04,\alpha=2.0$}
\label{fig34}
\end{minipage}
\end{figure}
The increasing dissipative effect of increasing $\alpha$ is clear from the sequence of figures, Fig. (\ref{fig31})-(\ref{fig34}), with the latter case of $\alpha=2$ being the classic case.
 
\section{Concluding remarks}
We propose a discontinuous Galerkin method for convection-diffusion problems in which the a fractional diffusion is expressed through a fractional Laplacian. To obtain a high order accuracy, we rewrite the fractional Lapacian as a composite of first order derivatives and integrals and transform the problem to a low order system. We consider the equation in a domain $\Omega$ with homogeneous boundary conditions. A local discontinuous Galerkin method is proposed and stability and error estimations are derived, establishing optimal convergence. In the numerical examples, the analysis is confirmed for different values of  $\alpha$.

\section*{Acknowledgments}
The first author was supported by the China Scholarship Council (No. 2011637083) and the Hunan Provincial Innovation Foundation for Postgraduates (No. CX2011B080). The second author was partially supported by the NSF DMS-1115416 and by OSD/AFOSR FA9550-09-1-0613.

\end{document}